\documentclass[12pt]{aptpub}

\usepackage{amsmath,amstext,url}

\renewcommand{\shorttitle}

\numberwithin{equation}{section} 

\makeatletter \@addtoreset{equation}{section}

\makeatletter \@addtoreset{lemma}{section}

\makeatletter \@addtoreset{theorem}{section}

\makeatletter \@addtoreset{corollary}{section}

\makeatletter \@addtoreset{remark}{section}

\makeatletter \@addtoreset{definition}{section}

\makeatletter \@addtoreset{example}{section}

\begin{document}

\thispagestyle{firstpg}
 \vspace*{1.5pc} \begin{center}
\normalsize{\bf{The $M^X/M/c$ queue with state-dependent control at
idle time and catastrophes}}\end{center} \hfill

\centerline{\scshape  $^1$Junping Li $\quad$    $^{2,1}$Lina Zhang}
\medskip
{\footnotesize
 \centerline{$^1$School of Mathematics and Statistics,}
   \centerline{Central South University, Changsha $410083$, Hunan Province, P.R. China}
 \centerline{$^2$School of Mathematics and Computational Science,}
\centerline{Xiangtan University,   Xiangtan 411105, Hunan Province,
P.R. China} }

\par
\footnote{\hspace*{-0.75pc}Email address: jpli@mail.csu.edu.cn, \
zln514@163.com}
\par
\renewenvironment{abstract}{%
\vspace{12pt} \vspace{1.5pc} \hspace*{2.25pc}
\begin{minipage}{14cm}
\footnotesize {\bf Abstract.} } {\end{minipage}}
\begin{abstract}
In this paper, we consider an $M^X/M/c$ queue with state-dependent
control at idle time and catastrophes. Properties of the queues
which terminate when the servers become idle are firstly studied.
Recurrence, equilibrium distribution and equilibrium queue-size
structure are studied for the case of resurrection and no
catastrophes. All of these results and the first effective
catastrophe occurrence time are then investigated for the case of
resurrection and catastrophes. In particular, we can obtain the
Laplace transform of the transition probability for the absorptive
$M^X/M/c$ queue.
\end{abstract}

\vspace*{12pt} \hspace*{2.25pc}
\parbox[b]{33.05pc}{{
}}
\par
\normalsize

\renewcommand{\amsprimary}[1]{
     \vspace*{8pt}
     \hspace*{2.25pc}
     \parbox[b]{24.75pc}{\scriptsize
    AMS 2000 Subject Classification: Primary 60J35 Secondary
    60J27; 60K25
     {\uppercase{#1}}}\par\normalsize}
\renewcommand{\ams}[2]{
     \vspace*{8pt}
     \hspace*{2.25pc}
     \parbox[b]{24.75pc}{\scriptsize
     AMS 2000 SUBJECT CLASSIFICATION: PRIMARY
     {\uppercase{#1}}\\ \phantom{
     AMS 2000
     SUBJECT CLASSIFICATION:
     }
    SECONDARY
 {\uppercase{#2}}}\par\normalsize}

\ams{60J35}{60J27, 60K25}

\par
\vspace{5mm}
 \setcounter{section}{1}
 \setcounter{equation}{0}
 \setcounter{theorem}{0}
 \setcounter{lemma}{0}
 \setcounter{corollary}{0}

\noindent {\large \bf 1. Introduction} \vspace{3mm}
\par
Markovian queue theory is a basic and important branch of queueing
theory. It interweaves general theory of queueing models and general
theory and applications of continuous time Markov chains and has
become a very successful and fruitful research field. Many research
works can be referenced, for example, Gross and
Harris~\cite{Gr1985}, Asmussen~\cite{As2003} for the former and
Anderson~\cite{An1991} for the latter. See also Chen~\cite{CMF92}
and~\cite{CMF04}, which contain many new materials for continuous
time Markov chains.
\par
 Markovian queueing models with state-independent and state-dependent controls have also attracted
considerable research interests. A most recent work on this
direction could be seen in Chen \textit{et al} \cite{CPLZ2010a}. In
these models, arbitrary inputs are allowed when the queue is empty.
Usually, this is due to the consideration of improving working
efficiency. Gelenbe \cite{Ge1991} and Gelenbe \textit{et al}
\cite{Geetal1991} introduced the particularly interesting concept of
negative arrivals, whilst other related papers include Jain and
Sigman \cite{Jain1996} and Bayer and Boxma \cite{Bayer1996}.
Parthasarathy and Krishna Kumar \cite{Pa1991} allowed arbitrary
input when the queue is empty, and Chen and Renshaw
\cite{Ch1997,Ch2004} introduced the possibility of removing the
entire workload. Di Crescenzo \textit{et al} \cite{DGNR2008}
considered the first effective catastrophe occurrence time of a
birth-death process. Dudin and Karolik \cite{DK2001} investigated a
BMAP/SM/1 system which is exposed to disasters arrivals. The
$M_t/M_t/N$ queue with catastrophes can be seen in Zeifman and
Korotysheva \cite{ZK2012}. From practical point of view, models with
disasters are quite interesting, e.g., the work of hardware
influenced by breaks and occasional power disappearance, the work of
communication systems influenced by computer viruses or intentional
external interventions, deletion of transactions in databases, the
operation of air defense radars, etc. More detailed information
about the real and potential applications and corresponding
descriptions can be found in \cite{AR2000}.
\par
It is worth noting that the papers \cite{CPLZ2010a, Ch1997, Ch2004}
discussed the queues with single server in the system. The present
paper considers $c$-servers in the system, it is a natural
generalization of the models considered in \cite{CPLZ2010a, Ch1997,
Ch2004}. Since there are more than one servers in the system, the
method in \cite{CPLZ2010a, Ch1997, Ch2004} fails and we have to find
some other techniques and methods to treat it. Moreover, in order to
consider the recurrence properties, equilibrium distribution,
equilibrium queue-size structure and the first effective catastrophe
occurrence time of the modified $M^X/M/c$ queue, we have to first
show all the Laplace transforms of the transition probability for
the absorptive $M^X/M/c$ queue, as usual, this is very difficult to
get the transition probability or transition function for a general
Markov process.
\par
We shall define our model by specifying the infinitesimal
characteristic, i.e., the so called $q$-matrix. Throughout this
paper, let ${\bf{E}}=\{0,1,2,\cdots\}$ denote the set of nonnegative
integers.
\begin{definition}\label{Definition 1.1}
A conservative $q$-matrix $Q=(q_{ij};\ i,j\in {\bf{E}})$ is called
the $M^X/M/c$ queueing $q$-matrix with state-dependent control at
idle time and catastrophes (henceforth referred to as the modified
$M^X/M/c$ $q$-matrix), if
\begin{equation}\label{bbb 1.1}
Q=Q^*+Q_s+Q_d,
\end{equation}
where $Q^*=(q_{ij}^*;\ i,j\in {\bf{E}})$,  $Q_s=(q_{ij}^{(s)};\
i,j\in {\bf{E}})$ and $Q_d=(q_{ij}^{(d)};\
 i,j\in {\bf{E}})$ are all conservative $q$-matrices which are given as
follows
\begin{equation}\label{bbb 1.2}
q_{ij}^*
   =\begin{cases}
     \min(i,c)b_{0}, &  \mbox{if $i\geq 1,\ j= i-1$},\\
     b_{1}-[\min(i,c)-1]b_0, & \mbox{if $i\geq 1,\ j= i$},\\
     b_{j-i+1}, & \mbox{if $i\geq 1,\ j\geq i+1$},\\
     0,              & \mbox{otherwise},
    \end{cases}
\end{equation}
 \begin{equation}\label{qsdefinition}
q_{ij}^{(s)}
   =\begin{cases}
    -h, & \mbox{if $i=0,\ j=0$},\\
     h_j, &  \mbox{if $i=0,\ j\geq 1$},\\
     0,              & \mbox{otherwise},
    \end{cases}
\end{equation}
\begin{equation}\label{bbb 1.3}
q_{ij}^{(d)}
   =\begin{cases}
     \beta, & \mbox{if $i\geq 1,\ j= 0$},\\
     -\beta, & \mbox{if $i\geq 1,\ j=i$},\\
     0,              & \mbox{otherwise},
    \end{cases}
\end{equation}
respectively. Here $\beta\geq0,\ h_j\geq0\ (j\geq1)$ and $ b_j\geq0\
(j\neq1)$ with
\begin{equation*}
0\leq h:=\sum_{j=1}^\infty h_j<\infty\ \ and\ \
0<-b_1=\sum_{j\neq1}b_j<\infty.
\end{equation*}
\end{definition}
\par
Since the modified $M^X/M/c$ $q$-matrix $Q$ is conservative and
bounded, by the theory of Markov chains(see \cite{An1991}), there
exists a unique $Q$-process, i.e., Feller minimal $Q$-process, which
satisfies the Kolmogorov forward and backward equations. We call
this process  the modified $M^X/M/c$ queueing process denoted by
$\{X_t;\ t\geq0\}$. Furthermore, denote
\begin{equation}\label{tidleQ}
\tilde{Q}=Q^*+Q_s,
\end{equation}
and let $\{\tilde{X}_t;\ t\geq0\}$ be the corresponding
$\tilde{Q}$-process. In order to avoid discussing some trivial
cases, throughout this paper we assume that $b_0>0$ and
$\sum_{j=2}^{\infty}b_j>0$.
\par
This process includes several interesting models as special cases.
For example, if $c=1$, we recover the model considered in Chen and
Renshaw \cite{Ch2004}. Whilst if $c=1$ and
 $\beta=0\ $(i.e., no catastrophe) and $h_j\equiv b_{j+1}\ (j\geq1),$ we obtain the ordinary $M^X/M/1$ queue.
 If $c=1$, $\beta=0$, $h_1=b_2$ and $h_j=b_{j+1}=0\ (j\geq2)$, we obtain the simple $M/M/1$ queue. Whilst taking
  $c=1$ and $b_j=0\ (j\geq3)$
 generates the model considered in Chen and Renshaw
  \cite{Ch1997}. Furthermore, if $\beta=0$, $h_1=b_2$ and $h_j=b_{j+1}=0\ (j\geq2)$, we obtain the simple $M/M/c$ queue.
\par
The structure of this paper is organized as follows. We first
introduce some key lemmas and consider the properties of the stopped
$M^X/M/c$ queue in Section 2, since our later discussion depends
heavily upon them. The $M^X/M/c$ queue, modified to ensure that
arbitrary mass arrivals may occur  when the system is empty, is
fully discussed in Section 3. Whilst the $M^X/M/c$ queue with
resurrection and catastrophes is studied in Section 4.
\\\\
\setcounter{section}{2}
 \setcounter{equation}{0}
 \setcounter{theorem}{0}
 \setcounter{lemma}{0}
 \setcounter{corollary}{0}
\noindent {\large \bf 2. The stopped $M^X/M/c$ queue} \vspace{3mm}
\par
We first consider the $M^X/M/c$ queue, modified by terminating the
process when it first becomes empty. The associated $q$-matrix $Q^*$
is given by \eqref{bbb 1.2}. Define the generating functions
\begin{equation}\label{bbb 2.1}
B(s)=\sum_{j=0}^\infty b_js^j\ \mbox{and}\
B_i(s)=B(s)+(i-1)b_0(1-s),\ \ i=1,2,\cdots,c.
\end{equation}
\par
These are all well defined for $s\in[-1,1],$ whilst as power-series
they are $C^\infty$ functions with respect to $s$ on $(-1,1).$ Also,
$B(s)$ and $B_i(s)$ are convex functions on $[0,1].$ The convex
property of $B_c(s)$ immediately yields the following simple, yet
important, result.
\begin{lemma}\label{ablele 2.1}
The equation $B_c(s)=0$ has a smallest root $u$ on $[0,1]$ with
$u=1$ if $B_c^\prime(1)\leq 0$ and $u<1$ if $B_c^\prime(1)>0$. More
specifically,
\par
\emph{(i)} if $B_c^\prime(1)\leq 0$, then $B_c(s)>0$ for all
$s\in[0,1),$ whence $B_c(s)=0$ has exactly one root $s=1$ in
$[0,1]$\emph{;} and,
\par
\emph{(ii)} if $B_c^\prime(1)>0$, then $B_c(s)=0$ has exactly two
roots, $u$ and $1$, in $[0,1]$ such that $B_c(s)>0$ when $0\leq s<u$
and $B_c(s)<0$ when $u<s<1.$
\end{lemma}
\par
We now define, for any $\lambda>0,$
\begin{equation}\label{bbb 2.4}
U_\lambda(s):=B_c(s)-\lambda s.
\end{equation}
It is clear that $U_\lambda(\cdot)$ is also $C^\infty$ with respect
to $s$ on $(-1,1)$ and that it is convex on $[0,1].$ The following
result can be easily proved.
\begin{lemma}\label{abclele 2.2}
For any fixed $\lambda>0,$ the equation $U_\lambda(s)=0$ has exactly
one root $u(\lambda)$ on $[0,1],$ and $0<u(\lambda)<1.$
\end{lemma}
\par
Similar to the proof of Lemma 2.3 in \cite{Ch2004}, we can obtain
the following Lemma.
\begin{lemma}\label{abclele 2.3}
For $u(\cdot)$ as defined in Lemma~\emph{\ref{abclele 2.2}}.
\emph{(i)} $u(\lambda)\in C^\infty(0,\infty)$; \emph{(ii)}
$u(\lambda)$ is a decreasing function of $\lambda>0$; \emph{(iii)}
$u(\lambda)\downarrow0$ and $\lambda u(\lambda)\rightarrow cb_0$ as
$\lambda\rightarrow\infty$; \emph{(iv)} when $\lambda\rightarrow
0^+$,
\begin{equation}\label{bbb 2.5}
u(\lambda)\uparrow u
\begin{cases}
=1& \mbox{if $B_c^\prime(1)\leq 0$},\\
<1& \mbox{if $B_c^\prime(1)>0$},
\end{cases}
\end{equation}
 where $u$ is the smallest root of $B_c(s)=0$ on
 $[0,1]$;
\par
\emph{(v)} for any positive integer $k$,
\begin{equation}\label{bbb 2.6}
\lim_{\lambda\rightarrow0^+}\frac{1-u(\lambda)^k}{\lambda}=
\begin{cases}
\infty& \mbox{if $B_c^\prime(1)\geq0$},\\
\frac{k}{-B_c^\prime(1)}& \mbox{if $B_c^\prime(1)<0$}.
\end{cases}
\end{equation}
\end{lemma}
\par
Let $(p_{ij}^*(t);\ i,j\geq0)$ and $(\phi_{ij}^*(\lambda);\
i,j\geq0)$ be the $Q^*$-function and $Q^*$-resolvent, respectively.
\par
\begin{theorem}\label{the resolvent of transition probability}
For any $i\geq0$, $(\phi_{ij}^*(\lambda);\ 0\leq j\leq c-1)$ is the
unique solution of the following linear equations
\begin{equation}\label{lamdaphii0starlamda}
\begin{cases}
-\lambda
\phi_{i0}^*(\lambda)-\sum_{k=1}^{c-1}u(\lambda)^{k-1}[B_c(u(\lambda))-B_k(u(\lambda))]\phi_{ik}^*(\lambda)
=-u(\lambda)^i,\\
-\lambda
\phi_{i0}^*(\lambda)+b_0\phi_{i1}^*(\lambda)=-\delta_{i0},\\
(b_1-\lambda)\phi_{i1}^*(\lambda)+2b_0\phi_{i2}^*(\lambda)=-\delta_{i1},\\
\ \ \ \ \ \ \cdots\\
\sum_{k=1}^{j-1}\phi_{ik}^*(\lambda)b_{j-k+1}+[b_1-(j-1)b_0-\lambda]\phi_{ij}^*(\lambda)
+(j+1)b_0\phi_{ij+1}^*(\lambda)=-\delta_{ij},\\
\ \ \ \ \ \ \cdots\\
\sum_{k=1}^{c-3}\phi_{ik}^*(\lambda)b_{c-k-1}+[b_1-(c-3)b_0-\lambda]\phi_{ic-2}^*(\lambda)
+(c-1)b_0\phi_{ic-1}^*(\lambda)=-\delta_{ic-2},
\end{cases}
\end{equation}
where $u(\lambda)(\lambda>0)$ is the unique root of $U_\lambda(s)=0$
on $[0,1]$.
\end{theorem}
\begin{proof}
By the Kolmogorov forward equations we have {\small
\begin{equation}\label{phiijlamda}
\begin{cases}
-\lambda
\phi_{i0}^*(\lambda)+b_0\phi_{i1}^*(\lambda)=-\delta_{i0},\\
(b_1-\lambda)\phi_{i1}^*(\lambda)+2b_0\phi_{i2}^*(\lambda)=-\delta_{i1},\\
\ \ \ \ \ \ \cdots\\
\sum_{k=1}^{c-1}\phi_{ik}^*(\lambda)b_{c-k+1}+[b_1-(c-1)b_0-\lambda]\phi_{ic}^*(\lambda)
+cb_0\phi_{ic+1}^*(\lambda)=-\delta_{ic},\\
\sum_{k=1}^{j-1}\phi_{ik}^*(\lambda)b_{j-k+1}+[b_1-(c-1)b_0-\lambda]\phi_{ij}^*(\lambda)
+cb_0\phi_{ij+1}^*(\lambda)=-\delta_{ij},\ \ \ (j\geq c+1).
\end{cases}
\end{equation}}
 Thus,
\begin{equation*}
\lambda\sum_{j=0}^\infty
\phi_{ij}^*(\lambda)s^j-s^i=\sum_{k=1}^{c-1}s^{k-1}B_k(s)\phi_{ik}^*(\lambda)+B_c(s)\sum_{j=c}^\infty
\phi_{ij}^*(\lambda)s^{j-1},
\end{equation*}
Since $\lambda u(\lambda)=B_c(u(\lambda))$, we have
\begin{equation}\label{the condition of ulamda}
-\lambda
\phi_{i0}^*(\lambda)-\sum_{k=1}^{c-1}u(\lambda)^{k-1}[B_c(u(\lambda))-B_k(u(\lambda))]
\phi_{ik}^*(\lambda)=-u(\lambda)^i.
\end{equation}
Combining \eqref{the condition of ulamda} with the first $c-1$
equations of \eqref{phiijlamda} we can  obtain the unique
$(\phi_{ij}^*(\lambda);\ i\geq0,\ 0\leq j\leq c-1)$. The proof is
complete.
\end{proof}
\begin{remark}\label{remark 2.1}
As usual, it is very difficult to get the transition probability or
transition function $(p_{ij}^*(t);i,j\in {\bf E})$ for a general
Markov process. We find it particularly noteworthy that we can
obtain the resolvent $(\phi_{ij}^*(\lambda);\ i\geq0,\ 0\leq j\leq
c-1)$ of the transition probability $(p_{ij}^*(t);\ i\geq0,\ 0\leq
j\leq c-1)$ from Theorem~\ref{the resolvent of transition
probability}. Furthermore, by carefully checking the proof of
Theorem~\ref{the resolvent of transition probability} and using
\eqref{phiijlamda}, we can obtain all the resolvent
$(\phi_{ij}^*(\lambda);\ i,j\in {\bf E})$ of the transition
probability $(p_{ij}^*(t); i,j\in {\bf E})$.
\end{remark}
\par
By the Kolmogorov forward equation $\Phi^*(\lambda)(\lambda
I-Q^*)=I$ and Lemma~\ref{abclele 2.2}, we can obtain the following
theorem.
\begin{theorem}\label{abthth 2.1}
The generating functions of the $Q^*$-resolvent take the form
\begin{align}\label{bbb 2.13}
L_i(\lambda,s)&:=\sum_{j=0}^\infty\phi_{ij}^*(\lambda)s^j\nonumber \\
&=\frac{B_c(s)\phi_{i0}^*(\lambda)+\sum_{k=1}^{c-1}\phi_{ik}^*(\lambda)s^k[B_c(s)-B_k(s)]-s^{i+1}}{U_\lambda(s)}\
\ (i\geq0),
\end{align}
where $U_\lambda(s)$ is defined by \eqref{bbb 2.4}, and
$(\phi_{ik}^*(\lambda); i\geq0,\ 1\leq k\leq c-1)$ is given by
Theorem~\emph{\ref{the resolvent of transition probability}}. In
particular, $\phi_{00}^*(\lambda)=\frac{1}{\lambda},$ whilst, for
$i,j\geq1,$
\begin{equation}\label{bbb 2.14}
\phi_{i0}^*(\lambda)=\frac{u(\lambda)^i-\sum_{k=1}^{c-1}\phi_{ik}^*(\lambda)u(\lambda)^{k-1}(c-k)
b_0(1-u(\lambda))}{\lambda}\ \ \mbox{and}\ \ \phi_{0j}^*(\lambda)=0,
\end{equation}
where $u(\lambda)(\lambda>0)$ is the unique root of $U_\lambda(s)=0$
on $[0,1]$.
\end{theorem}
\par
Similar to that considered in Lemma 3.1 in Li and Chen
\cite{Li2006}, we immediately get the following lemma.
\begin{lemma}
\label{abclele 2.5} For any $i\geq0$, we have
\begin{equation}\label{aa 2.4}
\int_0^\infty p_{ik}^*(t)dt<\infty,\ \ \ k\geq1.
\end{equation}
\end{lemma}
\par
Let $\{X^*_t;\ t\geq0\}$ denote the $Q^*$-process. Define the
extinction time as
\begin{equation*}
\tau^*_0:=
\begin{cases}
\inf\{t>0;\ X^*_t=0\}& \mbox{if $X^*_t=0$ for\ some\ $t>0$},\\
\infty& \mbox{if $X^*_t>0$ for\ all $t>0$}
\end{cases}
\end{equation*}
 and $w^*_k(t):=P(\tau^*_0\leq t|X^*_0=k).$ From \eqref{bbb 2.14}, we can immediately get the following
conclusion.
\begin{lemma}\label{abcthth 2.2}
Suppose that the $Q^*$-process $\{X^*_t;\ t\geq0\}$ starts from
$X^*_0=k>0$. Then the Laplace transform of $w^*_k(t)$ is
\begin{equation*}\label{bbb 2.15}
\int_0^\infty e^{-\lambda t}P(\tau^*_0\leq
t|X^*_0=k)\emph{d}t=\frac{u(\lambda)^k-\sum_{i=1}^{c-1}\phi_{ki}^*(\lambda)u(\lambda)^{i-1}
(c-i)b_0(1-u(\lambda))}{\lambda},
\end{equation*}
where $u(\lambda)$ is the unique root of $U_\lambda(s)=0$ on $[0,1]$
and thereby possesses the properties in Lemma~\emph{\ref{abclele
2.3}}, $(\phi_{ki}^*(\lambda);\ k\geq1,\ 1\leq i\leq c-1)$ can be
obtained by Theorem~\emph{\ref{the resolvent of transition
probability}}.
\end{lemma}
\par
\begin{theorem}\label{the extinction probability}
For the $Q^*$-process $\{X^*_t;\ t\geq0\}$, denote the extinction
probability
$e^*_k=P(\tau^*_0<\infty|X^*_0=k)=\lim\limits_{t\rightarrow\infty}p_{k0}^*(t)\
\ (k\geq1)$ and $m^*_i(k)=\int_0^\infty p_{ki}^*(t)dt\ (i\geq1)$.
\par
\emph{(i)}\ If $B_c^\prime(1)\leq 0$, then for any $k\geq 1$,
$e^*_k=1$\emph{;}
\par
\emph{(ii)}\ if $B_c^\prime(1)>0$, then for any $k\geq1$, $e^*_k$
and $(m^*_i(k);\ 1\leq i\leq c-1)$ is the unique solution of the
following linear equations
\begin{equation}\label{eklinearequations}
\begin{cases}
e^*_k=u^k-\sum_{i=1}^{c-1}m^*_i(k)u^{i-1}(c-i)b_0(1-u),\\
b_0m^*_1(k)=e^*_k,\\
b_1m^*_1(k)+2b_0m^*_2(k)=-\delta_{k1},\\
\ \ \ \ \cdots\\
\sum_{i=1}^{j-1}b_{j-i+1}m^*_i(k)+[b_1-(j-1)b_0]m^*_j(k)+(j+1)b_0m^*_{j+1}(k)=-\delta_{kj},\\
\ \ \ \ \cdots\\
\sum_{i=1}^{c-3}b_{c-i-1}m^*_i(k)+[b_1-(c-3)b_0]m^*_{c-2}(k)+(c-1)b_0m^*_{c-1}(k)=-\delta_{kc-2},
\end{cases}
\end{equation}
where $u$ is the smallest root of the equation $B_c(s)=0$ on $[0,1]$
with $u=1\ \mbox{if} \ B_c^\prime(1)\leq 0\ \ \mbox{and}\ \ u<1\
\mbox{if}\ B_c^\prime(1)>0$. Moreover, all the $(m^*_i(k);\ k\geq1,\
i\geq1)$ can be obtained.
\par
\emph{(iii)}\ The mean extinction time is
\begin{equation*}\label{bbb 2.18}
E(\tau^*_0|X^*_0=k)=
\begin{cases}
-\frac{1}{B_c^\prime(1)}[k+\sum_{i=1}^{c-1}m^*_i(k)(c-i)b_0]&\ \mbox{if $B_c^\prime(1)< 0$},\\
\infty&\ \mbox{if $B_c^\prime(1)\geq0$},
\end{cases}
\end{equation*}
where $(m^*_i(k);\ 1\leq i\leq c-1)$ is given by
\eqref{eklinearequations}.
\end{theorem}
\begin{proof}
Using Lemma~\ref{abclele 2.5}, Lemma~\ref{abcthth 2.2}, \eqref{bbb
2.5}, \eqref{aa 2.4}, in combination with the Tauberian theorem,
yields
\begin{align}\label{bbb 2.19}
e^*_k&=\lim_{\lambda\rightarrow0^+}\lambda \phi_{k0}^*(\lambda)
=\lim_{\lambda\rightarrow0^+}\Big(u(\lambda)^k-\sum_{i=1}^{c-1}\phi_{ki}^*(\lambda)
u(\lambda)^{i-1}(c-i)b_0(1-u(\lambda))\Big)\nonumber \\
&=
\begin{cases}
1,&\ \mbox{if $B_c^\prime(1)\leq 0$},\\
u^k-\sum_{i=1}^{c-1}m^*_i(k)u^{i-1}(c-i)b_0(1-u)<1,&\ \mbox{if
$B_c^\prime(1)>0$},
\end{cases}
\end{align}
Thus (i) is proven. By \eqref{bbb 2.5}, \eqref{aa 2.4} and letting
$\lambda\rightarrow0^+$ in every equation of
\eqref{lamdaphii0starlamda}, we can immediately obtain
\eqref{eklinearequations}. Combining \eqref{eklinearequations} with
\eqref{phiijlamda} we can obtain all the $(m^*_i(k);\ k\geq1,\
i\geq1)$, which completes the proof of (ii).
\par
Using the Tauberian theorem once again and in conjunction with the
result \eqref{bbb 2.5}, \eqref{bbb 2.6} of Lemma~\ref{abclele 2.3}
we obtain
\begin{align*}
E(\tau^*_0|X^*_0=k)&=\int_0^\infty
P(\tau^*_0>t|X^*_0=k)\emph{d}t=\lim_{\lambda\rightarrow0^+}
\frac{1-\lambda\phi_{k0}^*(\lambda)}{\lambda}\nonumber \\
&=
\begin{cases}
-\frac{1}{B_c^\prime(1)}[k+\sum_{i=1}^{c-1}m^*_i(k)(c-i)b_0]&\ \mbox{if $B_c^\prime(1)< 0$},\\
\infty&\ \mbox{if $B_c^\prime(1)\geq0$},
\end{cases}
\end{align*}
thus, (iii) is proven. The proof is complete.
\end{proof}
\par

 \setcounter{section}{3}
 \setcounter{equation}{0}
 \setcounter{theorem}{0}
  \setcounter{remark}{0}
 \setcounter{lemma}{0}
 \setcounter{corollary}{0}
\noindent {\large \bf 3. The $M^X/M/c$ queue with resurrection}
\vspace{3mm}
\par
We now extend the process by allowing mass arrivals of size $j$ to
occur at rate $h_j$ when the queue is empty and $\beta=0$. So the
 $q$-matrix is $\tilde{Q}=Q^*+Q_s$, where $Q_s$ is given by
\eqref{qsdefinition}. In addition to the generating functions $B(s)$
and $B_i(s)$ defined in \eqref{bbb 2.1}, we need to define
\begin{equation}\label{bbb 3.1}
H(s):=\sum_{j=1}^\infty h_js^j.
\end{equation}
Obviously, $H(s)$ is well defined on $[-1,1]$ and $H(1)=h>0$.
Moreover, denote
\begin{equation}\label{bbb 3.2}
\mu_1=H^\prime(1)=\sum_{j=1}^\infty jh_j,
\end{equation}
 note that $\mu_1$ is not currently presumed to be finite.
\par
By Theorem~\ref{the extinction probability}, we have the following
corollary.
\begin{corollary}\label{busy period}
For the $M^X/M/c$ queue with resurrection, the mean busy period is
\begin{equation*}
B=\begin{cases}-\frac{1}{hB_c^\prime(1)}\sum_{k=1}^{\infty}h_k[k+\sum_{i=1}^{c-1}m^*_i(k)(c-i)b_0],\
& if\ {B'_c(1)<0}\\
\infty,\ & if \ {B'_c(1)\geq 0}.
\end{cases}
\end{equation*}
\end{corollary}

\par
Let $\tilde{P}(t)=(\tilde{p}_{ij}(t);\ i,j\geq0)$ and
$\tilde{R}(\lambda)=(\tilde{r}_{ij}(\lambda);\ i,j\geq0)$ be the
$\tilde{Q}$-function and $\tilde{Q}$-resolvent, respectively.
Similar to the proof of Theorem 3.1 in Chen and Renshaw
\cite{Ch2004}, using the resolvent decomposition theorem (see Chen
and Renshaw \cite{Ch1990}), we have the following conclusion.
\begin{theorem}\label{abcthth 3.1}
For $\tilde{R}(\lambda)=(\tilde{r}_{ij}(\lambda);\ i,j\geq0)$, we
have
\begin{align}\label{bbb 3.3}
\tilde{r}_{00}(\lambda)&=\Big[\lambda+\lambda\sum_{i=1}^\infty\sum_{j=1}^\infty
h_i\phi_{ij}^*(\lambda)\Big]^{-1},\\ \label{bbb 3.4}
\tilde{r}_{i0}(\lambda)&=\tilde{r}_{00}(\lambda)b_0\phi_{i1}^*(\lambda)\
\ \ (i\geq1),\\ \label{bbb 3.5}
\tilde{r}_{0j}(\lambda)&=\tilde{r}_{00}(\lambda)\sum_{i=1}^\infty
h_i\phi_{ij}^*(\lambda)\ \ \ (j\geq1),\\ \label{bbb 3.6}
\tilde{r}_{ij}(\lambda)&=\phi_{ij}^*(\lambda)+\tilde{r}_{i0}(\lambda)\sum_{k=1}^\infty
h_k\phi_{kj}^*(\lambda)\ \ \ (i,j\geq1),
\end{align}
where $\Phi^*(\lambda)=(\phi_{ij}^*(\lambda);\ i,j\geq0)$ is the
$Q^*$-resolvent given in Theorem~\emph{\ref{abthth 2.1}}.
\end{theorem}
\begin{remark}\label{remark 3.1}
The Laplace transforms $(\tilde{r}_{ij}(\lambda);\ i,j\geq0)$ of the
transition probability $(\tilde{p}_{ij}(t),$\\$\ i,j\geq0)$ can be
obtained. Indeed, by Theorem~\ref{the resolvent of transition
probability} and Remark~\ref{remark 2.1}, we can get
$(\phi_{ij}^*(\lambda);\ i,j\geq1)$. Then by \eqref{bbb 3.3} and
\eqref{bbb 3.4}, we  have
\begin{equation}\label{tildetilder00lambdar00lambda}
\tilde{r}_{00}(\lambda)=\bigg[\lambda+h-H(u(\lambda))+\sum_{k=1}^{c-1}u(\lambda)^{k-1}(c-k)
b_0(1-u(\lambda))\sum_{i=1}^\infty
h_i\phi_{ik}^*(\lambda)\bigg]^{-1}
\end{equation}
and
\begin{equation}\label{tildetilderi0lambdari0lambda}
\tilde{r}_{i0}(\lambda)=\frac{b_0\phi_{i1}^*(\lambda)}{\lambda+h-H(u(\lambda))+\sum_{k=1}^{c-1}
u(\lambda)^{k-1}(c-k)b_0(1-u(\lambda))\sum_{i=1}^\infty
h_i\phi_{ik}^*(\lambda)},\ \ i\geq 1.
\end{equation}
Substituting $\tilde{r}_{00}(\lambda)$ and $(\phi_{ij}^*(\lambda);\
i,j\geq1)$ into  \eqref{bbb 3.5}, we can obtain
$\tilde{r}_{0j}(\lambda)\ (j\geq1)$. Finally, making use of
$\tilde{r}_{i0}(\lambda)\ (i\geq1)$ and $(\phi_{ij}^*(\lambda);\
i,j\geq1)$ in \eqref{bbb 3.6}, we can get
$(\tilde{r}_{ij}(\lambda);\ i,j\geq1)$.
\end{remark}
\par
We now consider the recurrence properties of the modified $M^X/M/c$
queue determined by $q$-matrix $\tilde{Q}$.
\par
\begin{theorem}\label{abcthth 3.2}
For the modified $M^X/M/c$ queueing process with $q$-matrix
$\tilde{Q}$, we have
\par
\emph{(i)} the process is recurrent if and only if
$B_c^\prime(1)\leq0$,
\par
\emph{(ii)} the process is positive recurrent if and only if
$B_c^\prime(1)<0$ and $\mu_1<\infty$, where $\mu_1$ is defined by
\eqref{bbb 3.2}.
\end{theorem}
\begin{proof}
(i) Since the process is irreducible, it is recurrent if and only if
$\lim_{\lambda\rightarrow0^+}\tilde{r}_{00}(\lambda)=\infty$, and so
by \eqref{bbb 3.3} if and only if
\begin{equation}\label{bbb 3.7}
\lim_{\lambda\rightarrow0^+}\sum_{i=1}^\infty h_i\sum_{j=1}^\infty
\lambda\phi_{ij}^*(\lambda)=0.
\end{equation}
 So from \eqref{bbb 2.14} we see that \eqref{bbb 3.7} holds true if and only if
\begin{equation*}
\lim_{\lambda\rightarrow0^+}\Big(u(\lambda)^i-\sum_{k=1}^{c-1}\phi_{ik}^*(\lambda)u(\lambda)^{k-1}
(c-k)b_0(1-u(\lambda))\Big)=1\ \ \mbox{for all}\ \ i\geq1,
\end{equation*}
which, by \eqref{bbb 2.19}, is equivalent to $B_c^\prime(1)\leq0$.
\par
(ii) Again using irreducibility and \eqref{bbb 3.3}, the process is
positive recurrent if and only if
 $\lim_{\lambda\rightarrow0^+}\lambda \tilde{r}_{00}(\lambda)>0$, i.e.,
$\lim_{\lambda\rightarrow0^+}\sum_{i=1}^\infty h_i\sum_{j=1}^\infty
\phi_{ij}^*(\lambda)<\infty$. Comparison with \eqref{bbb 2.5},
\eqref{bbb 2.6}, \eqref{bbb 2.14} and \eqref{aa 2.4} then shows that
the process is positive recurrent if and only if $B_c^\prime(1)<0$
and $\mu_1=\sum_{i=1}^\infty ih_i<\infty$. The proof is complete.
\end{proof}
\par
Having determined conditions for our modified $M^X/M/c$ queue to be
positive recurrent, we are now in a position to determine the
equilibrium distribution by giving the generating function
$\tilde{\Pi}(s):=\sum_{j=0}^\infty\tilde{\pi}_js^j.$
\begin{theorem}\label{abcthth 3.3}
Suppose theat $B_c^\prime(1)<0$ and $\mu_1<\infty$. The equilibrium
generating function $\tilde{\Pi}(s)$ takes the form
\begin{equation}\label{bbb 3.10}
\tilde{\Pi}(s)=\tilde{\pi}_0\Big[1+\frac{s(h-H(s))}{B_c(s)}\Big]+\frac{1}{B_c(s)}\sum_{k=1}^{c-1}\tilde{\pi}_ks^k(c-k)b_0(1-s),
\end{equation}
where
\begin{equation}\label{bbb 3.11}
\begin{cases}
\tilde{\pi}_0=-B_c^\prime(1)\Big[-B_c^\prime(1)+\mu_1+\sum_{k=1}^{c-1}r_k(c-k)b_0\Big]^{-1},\\
\tilde{\pi}_k=\tilde{\pi}_0r_k\ \ (k\geq1),
\end{cases}
\end{equation}
and $r_k:=\sum_{i=1}^\infty h_im_k^*(i)\ (k\geq1)$ satisfies
\begin{equation}\label{rrklinearequations}
\begin{cases}
b_0r_1=h,\\
b_1r_1+2b_0r_2=-h_1,\\
\ \ \ \ \cdots\\
\sum_{i=1}^{c-1}b_{c-i+1}r_i+[b_1-(c-1)b_0]r_c+cb_0r_{c+1}=-h_c,\\
\sum_{i=1}^{j-1}b_{j-i+1}r_i+[b_1-(c-1)b_0]r_j+cb_0r_{j+1}=-h_{j}\ \
\ (j\geq c+1).
\end{cases}
\end{equation}
\end{theorem}
\begin{proof}
Noting that
$\tilde{\pi}_j=\lim\limits_{\lambda\rightarrow0^+}\lambda
\tilde{r}_{0j}(\lambda)$ for all $j\geq0$, let us first consider
$j=0$. Paralleling the proof of Theorem~\ref{abcthth 3.2} we see
that
\begin{equation*}
\tilde{\pi}_0=\lim_{\lambda\rightarrow0^+}\lambda
\tilde{r}_{00}(\lambda)=-B_c^\prime(1)\Big[-B_c^\prime(1)+\mu_1+\sum_{k=1}^{c-1}r_k(c-k)b_0\Big]^{-1},
\end{equation*}
whilst for $j\geq1$, it follows from \eqref{bbb 3.5} that
\begin{equation*}
\tilde{\pi}_j=\lim_{\lambda\rightarrow0^+}\lambda
\tilde{r}_{0j}(\lambda)=\tilde{\pi}_0\sum_{i=1}^\infty
h_i\int_0^\infty p_{ij}^*(t)dt=\tilde{\pi}_0r_j,
\end{equation*}
whence by \eqref{aa 2.4} and letting $\lambda\rightarrow0^+$ in
every equation of \eqref{phiijlamda}, we can immediately obtain
\eqref{rrklinearequations}. Thus, on applying \eqref{bbb 2.13} and
\eqref{bbb 2.14}, we have
\begin{align*}
\tilde{\Pi}(s)&=\tilde{\pi}_0\Big[1+\lim_{\lambda\rightarrow0^+}\sum_{i=1}^\infty
h_i\sum_{j=1}^\infty
\phi_{ij}^*(\lambda)s^j\Big]\\
&=\tilde{\pi}_0\Big[1+\frac{s(h-H(s))}{B_c(s)}\Big]+\frac{1}{B_c(s)}\sum_{k=1}^{c-1}\tilde{\pi}_ks^k(c-k)b_0(1-s).
\end{align*}
The proof is complete.
\end{proof}
\par
\begin{remark}
Theorem~\ref{abcthth 3.3} gives the generating function of
equilibrium distribution for $M^X/M/c$ queue with resurrection. If
$h_1=b_2$, $h_j=b_{j+1}=0\ (j\geq2)$ and let $\rho=\frac{b_2}{b_0}$,
then we recover the ordinary $M/M/c$ queue, in this case, \eqref{bbb
3.10} reduces to
\begin{equation*}
\tilde{\Pi}(s)=\frac{b_0}{cb_0-b_2s}\sum_{k=0}^{c-1}\tilde{\pi}_ks^k(c-k),
\end{equation*}
where
\begin{align}
\tilde{\pi}_0=\Big[\sum_{k=0}^c\frac{\rho^k}{k!}+\frac{\rho^{c+1}}{c!(c-\rho)}\Big]^{-1}\hspace{.5cm}
\mathrm{and} \hspace{.5cm}& \tilde{\pi}_k=
\begin{cases}
\frac{\rho^k}{k!}\tilde{\pi}_0& \mbox{$k=1,2,\ldots,c-1$},\\
\frac{\rho^k}{c^{k-c}c!}\tilde{\pi}_0&  \mbox{$k\geq c$}.\nonumber
\end{cases}
\end{align}
If $c=1$, then we recover the corresponding result in Chen and
Renshaw \cite{Ch2004}.
\end{remark}
\par
From \eqref{bbb 3.10}, we can obtain the following corollary which
illustrates other important queueing features.
\begin{corollary}
The equilibrium queue size, $N$, has expectation
\begin{equation}\label{EN}
E(N)=\tilde{\pi}_0\Big[\frac{B_c^{\prime\prime}(1)\mu_1}{2(B_c^\prime(1))^2}-\frac{2\mu_1
+H^{\prime\prime}(1)}{2B_c^\prime(1)}\Big]
+\sum_{k=1}^{c-1}\tilde{\pi}_k(c-k)b_0\Big[\frac{B_c^{\prime\prime}(1)}{2(B_c^\prime(1))^2}-\frac{k}{B_c^\prime(1)}\Big]
\end{equation}
if both $B_c^{\prime\prime}(1)$ and $H^{\prime\prime}(1)$ are
finite, and $E(N)=\infty$ otherwise. The equilibrium waiting queue
size, $L_w$, has expectation{\small
\begin{equation}\label{ELw}
E(L_w)=\tilde{\pi}_0\Big[\frac{B_c^{\prime\prime}(1)\mu_1}{2(B_c^\prime(1))^2}-\frac{2\mu_1+H^{\prime\prime}(1)}{2B_c^\prime(1)}
+c\Big]
+\sum_{k=1}^{c-1}\tilde{\pi}_k(c-k)\Big[b_0\Big(\frac{B_c^{\prime\prime}(1)}{2(B_c^\prime(1))^2}-\frac{k}{B_c^\prime(1)}\Big)
+1\Big]-c
\end{equation}}
if both $B_c^{\prime\prime}(1)$ and $H^{\prime\prime}(1)$ are
finite, and $E(L_w)=\infty$ otherwise. Here  $(\tilde{\pi}_k;\ 0\leq
k\leq c-1)$ is given in Theorem~\emph{\ref{abcthth 3.3}}.
\end{corollary}
\par
\setcounter{section}{4}
 \setcounter{equation}{0}
 \setcounter{theorem}{0}
 \setcounter{remark}{0}
 \setcounter{lemma}{0}
 \setcounter{corollary}{0}
\noindent {\large \bf 4. The $M^X/M/c$ queue  with resurrection and
catastrophes} \vspace{3mm}
\par
\par
In this section, we consider the general case that $\beta>0$, the
$q$-matrix $Q$ now takes the form \eqref{bbb 1.1}, i.e.,
$Q=Q^*+Q_s+Q_d$, where $Q_s$ and $Q_d$ are defined by
\eqref{qsdefinition} and \eqref{bbb 1.3}. Let $\{X_t;t\geq 0\}$
denote the $Q$-process, $P(t)=(p_{ij}(t);\ i,j\geq0)$ and
$R(\lambda)=(r_{ij}(\lambda);\ i,j\geq0)$ be the $Q$-function and
$Q$-resolvent, respectively. Noting that the properties of this
process are substantially different from that in the case $\beta=0$
considered in Section 3. Using the resolvent decomposition theorem
(see Chen and Renshaw \cite{Ch1990}), we have the following theorem.
\par
\begin{theorem}\label{abcthth 4.1}
For $R(\lambda)=(r_{ij}(\lambda);\ i,j\geq0)$, we have
\begin{align}\label{bbb 4.1}
r_{00}(\lambda)&=\Big[\lambda+\lambda\sum_{i=1}^\infty\sum_{j=1}^\infty
h_i\phi_{ij}^*(\lambda+\beta)\Big]^{-1},\\ \label{bbb 4.2}
r_{i0}(\lambda)&=r_{00}(\lambda)\Big[b_0\phi_{i1}^*(\lambda+\beta)
+\beta\sum_{k=1}^\infty\phi_{ik}^*(\lambda+\beta)\Big]\ \ \
(i\geq1), \\\label{bbb 4.3}
r_{0j}(\lambda)&=r_{00}(\lambda)\sum_{i=1}^\infty
h_i\phi_{ij}^*(\lambda+\beta)\ \ \ (j\geq1), \\\label{bbb 4.4}
r_{ij}(\lambda)&=\phi_{ij}^*(\lambda+\beta)+r_{i0}(\lambda)\sum_{k=1}^\infty
h_k\phi_{kj}^*(\lambda+\beta)\ \ \ (i,j\geq1),
\end{align}
where $\Phi^*(\lambda)=(\phi_{ij}^*(\lambda);\ i,j\geq0)$ is the
$Q^*$-resolvent given in Theorem~\emph{\ref{abthth 2.1}}.
\end{theorem}
\par
Replacing $\lambda$ with $\lambda+\beta$ in \eqref{bbb 2.14},
 we have{\small
\begin{equation}\label{sumk=1toinftyphiik*}
\sum_{k=1}^\infty\phi_{ik}^*(\lambda+\beta)
=\frac{1-u(\lambda+\beta)^i+\sum_{k=1}^{c-1}\phi_{ik}^*(\lambda+\beta)u(\lambda+\beta)^{k-1}(c-k)
b_0(1-u(\lambda+\beta))}{\lambda+\beta}.
\end{equation}}
Since
\begin{equation*}
b_0\phi_{i1}^*(\lambda+\beta)=(\lambda+\beta)\phi_{i0}^*(\lambda+\beta)=u(\lambda+\beta)^i-M_i(\lambda),
\end{equation*}
where
\begin{equation*}
M_i(\lambda)=\sum_{k=1}^{c-1}\phi_{ik}^*(\lambda+\beta)u(\lambda+\beta)^{k-1}(c-k)b_0(1-u(\lambda+\beta)).
\end{equation*}
It follows from \eqref{bbb 4.2} that
\begin{equation}\label{bbb 4.5}
r_{i0}(\lambda) =r_{00}(\lambda)\frac{\lambda
u(\lambda+\beta)^i-\lambda M_i(\lambda)+\beta}{\lambda+\beta}.
\end{equation}
Moreover, \eqref{bbb 4.1} can be rewritten as
\begin{align}\label{bbb 4.6}
r_{00}(\lambda)
=& \Big[\lambda+\frac{\lambda}{\lambda+\beta}\Big(H(1)-H(u(\lambda+\beta))\nonumber\\
&+\sum_{k=1}^{c-1}u(\lambda+\beta)^{k-1}(c-k)b_0(1-u(\lambda+\beta))\sum_{i=1}^\infty
h_i\phi_{ik}^*(\lambda+\beta)\Big)\Big]^{-1}.
\end{align}
\par
 Paralleling Section 2, define the hitting
time $\tau_0=\inf\{t>0;\ X_t=0\}$ with $\tau_0=\infty$ if $X_t>0$
for all $t>0$, and let $e_k=P(\tau_0<\infty|X_0=k)$ and
$w_k(t)=P(\tau_0\leq t|X_0=k)$ denote the hitting probability and
distribution function of $\tau_0$, starting from $X_0=k$,
respectively.
\begin{theorem}\label{abcthth 4.2}
If  $h=0$ and $\beta>0$. Then for any $k\geq1,$
\begin{align}\label{bbb 4.7}
&\int_0^\infty e^{-\lambda t}P(\tau_0\leq
t|X_0=k)\emph{d}t\nonumber \\
&=
\frac{1}{\lambda+\beta}\Big(\frac{\beta}{\lambda}+u(\lambda+\beta)^k-
\sum_{i=1}^{c-1}\phi_{ki}^*(\lambda+\beta)u(\lambda+\beta)^{i-1}(c-i)b_0(1-u(\lambda+\beta))\Big),
\end{align}
where $u(\lambda)$ is the unique root of $U_\lambda(s)=0$ on $[0,1]$
and $(\phi_{ki}^*(\lambda);\ k\geq1,\ 1\leq i\leq c-1)$ is given by
Theorem~\emph{\ref{the resolvent of transition probability}}.
Moreover, $e_k=1\ (k=1,2,\ldots)$ and the mean extinction time is
finite and given by
\begin{equation}\label{bbb 4.9}
E(\tau_0|X_0=k)=\frac{1}{\beta}\Big[1-u(\beta)^k+\sum_{i=1}^{c-1}\phi_{ki}^*(\beta)u(\beta)^{i-1}(c-i)
b_0(1-u(\beta))\Big].
\end{equation}
\end{theorem}
\begin{proof}
 \eqref{bbb 4.7} is just \eqref{bbb 4.5},
since $r_{00}(\lambda)=\frac{1}{\lambda}$ in the current case.
Hence,
\begin{equation*}
e_k=\lim_{\lambda\rightarrow0^+}\frac{\beta+\lambda
u(\lambda+\beta)^k-\lambda
\sum_{i=1}^{c-1}\phi_{ki}^*(\lambda+\beta)u(\lambda+\beta)^{i-1}(c-i)b_0(1-u(\lambda+\beta))}{\lambda+\beta}=1.
\end{equation*}
By using \eqref{bbb 4.7} and the Tauberian theorem, we can obtain
\eqref{bbb 4.9}. The proof is complete.
\end{proof}
\par
We now consider the case $h>0$. We first give the following
important lemma which follows from Theorem~\ref{the resolvent of
transition probability}.
\begin{lemma}\label{l_k1leqkleqc-1}
For $(p_{ij}^*(t);\ i,j\geq0)$ and $(\phi_{ij}^*(\lambda);\
i,j\geq0)$ given in Section \emph{2}, denote
$L_j(\lambda)=\sum_{i=1}^\infty h_i\phi_{ij}^*(\lambda)$ $\
(j\geq0)$. Then $(L_j(\lambda);\ 0\leq j\leq c-1)$ is the unique
solution of the following linear equations
\begin{equation*}
\begin{cases}
-\lambda
L_0(\lambda)-\sum_{k=1}^{c-1}u(\lambda)^{k-1}[B_c(u(\lambda))-B_k(u(\lambda))]L_k(\lambda)=-H(u(\lambda)),\\
-\lambda L_0(\lambda)+b_0L_1(\lambda)=0,\\
(b_1-\lambda)L_1(\lambda)+2b_0L_2(\lambda)=-h_1,\\
\ \ \ \ \ \ \cdots\\
\sum_{k=1}^{j-1}L_k(\lambda)b_{j-k+1}+[b_1-(j-1)b_0-\lambda]L_j(\lambda)+(j+1)b_0L_{j+1}(\lambda)=-h_j,\\
\ \ \ \ \ \ \cdots\\
\sum_{k=1}^{c-3}L_k(\lambda)b_{c-k-1}+[b_1-(c-3)b_0-\lambda]L_{c-2}(\lambda)+(c-1)b_0L_{c-1}(\lambda)=-h_{c-2},
\end{cases}
\end{equation*}
where $u(\lambda)(\lambda>0)$ is the unique root of $U_\lambda(s)=0$
on $[0,1]$. Moreover, all the $(L_j(\lambda);\ j\geq 0)$ can be
obtained.
\end{lemma}
\begin{theorem}\label{abcthth 4.3}
If $h,\beta>0$, then the $Q$-process is always positive recurrent.
\end{theorem}
\begin{proof}
It follows from \eqref{bbb 4.6} that
$\lim_{\lambda\rightarrow0^+}r_{00}(\lambda)=\infty$, and so the
$Q$-process is recurrent. Noting $\lim_{t\rightarrow
\infty}p_{00}(t)=\lim_{\lambda\rightarrow 0^+}\lambda
r_{00}(\lambda)$ and again using \eqref{bbb 4.6} yield
\begin{equation}\label{bbb 4.10}
\lim_{t\rightarrow\infty}p_{00}(t)=\beta\Big[\beta+H(1)-H(u(\beta))+\sum_{k=1}^{c-1}L_k(\beta)u(\beta)^{k-1}
(c-k)b_0(1-u(\beta))\Big]^{-1}>0.
\end{equation}
The proof is complete.
\end{proof}
\par
The following theorem further gives the equilibrium distribution
$\{\pi_j;\ j\geq0\}$ of the $Q$-process.
\par
\begin{theorem}\label{abcthth 4.4}
 The equilibrium distribution of the $Q$-process is given by{\small
\begin{equation}\label{4pis4}
\Pi(s)=
\pi_0\Big[1+\frac{s(H(u(\beta))-H(s))}{U_\beta(s)}\Big]+\frac{\sum_{k=1}^{c-1}\pi_k(c-k)b_0[s^k(1-s)
-su(\beta)^{k-1}(1-u(\beta))]}{U_\beta(s)},
\end{equation}}
where $\Pi(s)=\sum_{k=0}^{\infty}\pi_ks^k$, $U_\beta(s)$ is defined
in \eqref{bbb 2.4}. Furthermore,
\begin{align}\label{bbb 4.12}
\pi_0&=\beta\Big[\beta+H(1)-H(u(\beta))+\sum_{k=1}^{c-1}L_j(\beta)u(\beta)^{k-1}(c-k)b_0(1-u(\beta))\Big]^{-1},\\
\pi_j&=\pi_0L_j(\beta)\ \ (j\geq1)\nonumber
\end{align}
and $(L_j(\beta);\ j\geq1)$ is given by
Lemma~\emph{\ref{l_k1leqkleqc-1}}.
\end{theorem}
\begin{proof}
First note that \eqref{bbb 4.12} is just \eqref{bbb 4.10}, whilst,
it follows from~\eqref{bbb 4.3} that for $j\geq1$,
\begin{equation}\label{pijwhenbeta}
\pi_j=\lim_{\lambda\rightarrow0^+}\lambda
r_{0j}(\lambda)=\pi_0L_j(\beta).
\end{equation}
Hence, using \eqref{bbb 2.13} and the above two equality we can get
\eqref{4pis4}. The proof is comlete.
\end{proof}
\par
As a direct consequence of Theorem~\ref{abcthth 4.4}, we have the
following corollary regarding queue size.
\begin{corollary}
The equilibrium queue size, $N$, has expectation
\begin{align*}
E(N)&=\pi_0\frac{(H(u(\beta))-h-\mu_1)(-\beta)-(H(u(\beta))-h)(B_c^\prime(1)-\beta)}{\beta^2}\nonumber \\
&\quad+\frac{\sum_{k=1}^{c-1}\pi_k(c-k)b_0[\beta+u(\beta)^{k-1}(1-u(\beta))B_c^\prime(1)]}{\beta^2}
\end{align*}
and the equilibrium waiting queue size, $L_w$, has expectation
\begin{align*}
E(L_w)&=\pi_0\Big[\frac{(H(u(\beta))-h-\mu_1)(-\beta)-(H(u(\beta))-h)(B_c^\prime(1)-\beta)}{\beta^2}+c\Big]\nonumber \\
&\quad+\frac{\sum_{k=1}^{c-1}\pi_k(c-k)[b_0(\beta+u(\beta)^{k-1}(1-u(\beta))B_c^\prime(1))+\beta^2]}{\beta^2}-c,
\end{align*}
where $(\pi_k;\ 0\leq k\leq c-1)$ is given in
Theorem~\emph{\ref{abcthth 4.4}}.
\end{corollary}
\par
From now on, we consider the first effective catastrophe occurrence
time of the $Q$-process $\{X_t;\ t\geq0\}$.
\par
 For all $j,n\in {\bf{E}}$ and $t>0$, $p_{jn}(t)$ satisfy the
following system of forward equations {\small
\begin{equation}\label{the forward equations of Nt}
\begin{cases}
p'_{j0}(t)=-(h+\beta)p_{j0}(t)+b_0p_{j1}(t)+\beta,\\
p'_{jn}(t)=h_np_{j0}(t)+\sum\limits_{k=1}^{n-1}b_{n-k+1}p_{jk}(t)\\
\hspace{1.5cm}+[b_1-(n-1)b_0-\beta]p_{jn}(t)+(n+1)b_0p_{j,n+1}(t),\ \ 1\leq n\leq c-1,\\
p'_{jn}(t)=h_np_{j0}(t)+\sum\limits_{k=1}^{n-1}b_{n-k+1}p_{jk}(t)+[b_1-(c-1)b_0-\beta]p_{jn}(t)+cb_0p_{j,n+1}(t),
 n\geq c.
\end{cases}
\end{equation}}
Here $\sum_{k=1}^{n-1}=0$ if $n=1$ as a notation. Obviously, for all
$j,n\in {\bf{E}}$ and $t>0$, the $\tilde{Q}$-function
$\tilde{p}_{jn}(t)$ satisfy the system of forward equations
\eqref{the forward equations of Nt} by setting $\beta=0$.
\par
The following lemma reveals that $(p_{ij}(t))$ can be expressed in
terms of $(\tilde{p}_{ij}(t))$.
\begin{lemma}\label{the link between Nt and hatNt}\emph{(}see \emph{\cite{Pakes1997}}\emph{)}
For all $j,n\in {\bf{E}},\ t>0$, we have
\begin{equation}\label{the link between pjnt and hatpjnt}
p_{jn}(t)=e^{-\beta t}\tilde{p}_{jn}(t)+\beta\int_0^te^{-\beta
\tau}\tilde{p}_{0n}(\tau)\emph{d}\tau,
\end{equation}
or
\begin{equation}\label{link between pi and hatpi}
r_{jn}(\lambda)=\tilde{r}_{jn}(\lambda+\beta)+\frac{\beta}{\lambda}\tilde{r}_{0n}(\lambda+\beta),\
\ \ \lambda>0.
\end{equation}
\end{lemma}
\par
Let $C_{j0}$ denote the first occurrence time of an effective
catastrophe when $X_0=j\ (j\in {\bf{E}})$, and let $d_{j0}(t)\
(t>0)$ be the density of $C_{j0}$. In order to investigate on the
features of $C_{j0}$, let us refer to a modified $M^X/M/c$ queue
with catastrophes that will be denoted as $\{M_t;\ t\geq0\}$. Its
behavior is identical to that of $\{X_t;\ t\geq0\}$ before hitting
$0$, the only difference is that the effect of a catastrophe from
state $n>0$ makes a jump from $n$ to the absorbing state $-1$. In
other words, $\{M_t;\ t\geq0\}$ is a modified $M^X/M/c$ queuing
process with catastrophes with state space ${\bf
S}=\{-1,0,1,\ldots\}$. Its $q$-matrix is
\begin{equation}\label{QMthe  definition of QM}
Q_M=\bar{Q}^*+\bar{Q}_s+\bar{Q}_d,
\end{equation}
 where $\bar{Q}^*=(\bar{q}_{ij}^*,\ i,j\in{\bf S})$ and
$\bar{Q}_s=(\bar{q}_{ij}^{(s)},\ i,j\in{\bf S})$ are given by
\begin{equation*}
\bar{q}^*_{ij}
   =\begin{cases}q^*_{ij},\ & \mbox{if\ $i,j\in {\bf E}$},\\
         0,              & \mbox{otherwise}
    \end{cases}
\end{equation*}
and
\begin{equation*}
\bar{q}^{(s)}_{ij} =\begin{cases}q^{(s)}_{ij},\ & \mbox{if\ $i,j\in
{\bf
E}$},\\
                            0, \ & \mbox{otherwise},
\end{cases}
\end{equation*}
respectively. $\bar{Q}_d=(\bar{q}_{ij}^{(d)},\ i,j\in{\bf S})$ is
given by
\begin{equation}\label{Qdthe  definition of Qd}
\bar{q}_{ij}^{(d)}
   =\begin{cases}
    \beta, &  \mbox{if~$i\geq1,\ j=-1,$}\\
-\beta, &  \mbox{if~$i\geq1,\ j=i,$}\\
     0,              & \mbox{otherwise}.
    \end{cases}
\end{equation}
\par
Let $H(t)=(h_{jn}(t),\ j,n\in {\bf S})$ and
$\eta(\lambda)=(\eta_{jn}(\lambda),\ j,n\in {\bf S})$ be the
 $Q_M$-function and $Q_M$-resolvent, respectively. For all $j\in
{\bf{E}},\ n\in {\bf S}$ and $t\geq0$, $h_{jn}(t)$ satisfy the
following system of forward equations {\small
\begin{equation}\label{Mtforward equation for
Mt} \hspace{-1.007mm}  \begin{cases}
h'_{j,-1}(t)=\beta(1-h_{j,-1}(t)-h_{j0}(t)),\\
h'_{j0}(t)=-hh_{j0}(t)+b_0h_{j1}(t),\\
h'_{jn}(t)=h_nh_{j0}(t)+\sum\limits_{k=1}^{n-1}b_{n-k+1}h_{jk}(t)\\
\hspace{1.5cm}+[b_1-(n-1)b_0-\beta]h_{jn}(t)+(n+1)b_0h_{j,n+1}(t),\   1\leq n\leq c-1,\\
h'_{jn}(t)=h_nh_{j0}(t)+\sum\limits_{k=1}^{n-1}b_{n-k+1}h_{jk}(t)+[b_1-(c-1)b_0-\beta]h_{jn}(t)+cb_0h_{j,n+1}(t),
 n\geq c.
\end{cases}
\end{equation}}
By the relation of $\{X_t; t\geq 0\}$ and $\{M_t; t\geq 0\}$, we
have
\begin{equation}\label{link between cj0 and Mt}
P(C_{j0}>t)\equiv\int_t^{+\infty}d_{j0}(\tau)\emph{d}\tau=\sum_{n=0}^{+\infty}h_{jn}(t)=1-h_{j,-1}(t),\
\ j\in {\bf{E}}.
\end{equation}
\par
In the following theorem we shall express the Laplace transform
$(\eta_{jn}(\lambda);\ j\in {\bf{E}},n\in {\bf S})$.
\begin{theorem}\label{etaj-1theorem}
 For all $j\in {\bf{E}}$ and $\lambda>0$, we have
\begin{align}\label{etaj-1}
&\eta_{j,-1}(\lambda)=\frac{\beta}{\lambda+\beta}\bigg[\frac{1}{\lambda}-
\frac{\tilde{r}_{j0}(\lambda+\beta)}{1-\beta\tilde{r}_{00}(\lambda+\beta)}\bigg],\\
\label{jnetajn}
&\eta_{jn}(\lambda)=\tilde{r}_{jn}(\lambda+\beta)+\beta\tilde{r}_{0n}(\lambda+\beta)
\frac{\tilde{r}_{j0}(\lambda+\beta)}{1-\beta\tilde{r}_{00}(\lambda+\beta)},\
\ n\geq 0,
\end{align}
where $(\tilde{r}_{jn}(\lambda+\beta);\ j,n\in {\bf{E}})$ is given
by Remark~\emph{\ref{remark 3.1}}.
\end{theorem}
\begin{proof} For $j=0$, taking the Laplace transform of the second to the fourth equations in
\eqref{Mtforward equation for Mt}, we obtain
  {\small
\begin{equation}\label{00MtLaplace forward equation for Mt}
\begin{cases}
(\lambda+h)\eta_{00}(\lambda)-1=b_0\eta_{01}(\lambda),\\
[\lambda-b_1+(n-1)b_0+\beta]\eta_{0n}(\lambda)=h_n\eta_{00}(\lambda)+\sum\limits_{k=1}^{n-1}b_{n-k+1}\eta_{0k}(\lambda)\\
\hspace{5.5cm}+(n+1)b_0\eta_{0,n+1}(\lambda),\   1\leq n\leq c-1,\\
[\lambda-b_1+(c-1)b_0+\beta]\eta_{0n}(\lambda)=h_n\eta_{00}(\lambda)+\sum\limits_{k=1}^{n-1}b_{n-k+1}\eta_{0k}(\lambda)
+cb_0\eta_{0,n+1}(\lambda),
 n\geq c.
\end{cases}
\end{equation}}
Similarly, taking the Laplace transform of \eqref{the forward
equations of Nt}, we have
 {\small
\begin{equation}\label{00Laplace forward equation for Nt}
\begin{cases}
(\lambda+h+\beta)r_{00}(\lambda)-1=b_0r_{01}(\lambda)+\frac{\beta}{\lambda}, \\
[\lambda-b_1+(n-1)b_0+\beta]r_{0n}(\lambda)=h_nr_{00}(\lambda)+\sum\limits_{k=1}^{n-1}b_{n-k+1}r_{0k}(\lambda)\\
\hspace{5.5cm}+(n+1)b_0r_{0,n+1}(\lambda),\  1\leq n\leq c-1,\\
[\lambda-b_1+(c-1)b_0+\beta]r_{0n}(\lambda)=h_nr_{00}(\lambda)+\sum\limits_{k=1}^{n-1}b_{n-k+1}r_{0k}(\lambda)
+cb_0r_{0,n+1}(\lambda),
 n\geq c.
\end{cases}
\end{equation}}
Let
\begin{equation}\label{etaonlambda}
\eta_{0n}(\lambda)=A(\lambda)r_{0n}(\lambda),\ \ n\geq 0.
\end{equation}
Substituting \eqref{etaonlambda} into \eqref{00MtLaplace forward
equation for Mt}, and using \eqref{00Laplace forward equation for
Nt}, we obtain
\begin{equation}\label{Alambda}
A(\lambda)=\frac{\lambda}{\lambda+\beta-\lambda\beta
r_{00}(\lambda)}.
\end{equation}
Making use of \eqref{link between pi and hatpi} in
\eqref{etaonlambda}, with $A(\lambda)$ given in \eqref{Alambda}, we
obtain \eqref{jnetajn} for $j=0$.
\par
For $1\leq j\leq c-1$, taking the Laplace transform of the second to
the fourth equations in \eqref{Mtforward equation for Mt}, we have
{\small
\begin{equation}\label{11MtLaplace forward equation for Mt}
\begin{cases}
(\lambda+h)\eta_{j0}(\lambda)=b_0\eta_{j1}(\lambda),\\
[\lambda-b_1+(n-1)b_0+\beta]\eta_{jn}(\lambda)-\delta_{jn}=h_n\eta_{j0}(\lambda)+\sum\limits_{k=1}^{n-1}b_{n-k+1}
\eta_{jk}(\lambda)\\
\hspace{6.5cm}+(n+1)b_0\eta_{j,n+1}(\lambda),\  1\leq n\leq c-1,\\
[\lambda-b_1+(c-1)b_0+\beta]\eta_{jn}(\lambda)=h_n\eta_{j0}(\lambda)+\sum\limits_{k=1}^{n-1}b_{n-k+1}\eta_{jk}
(\lambda)+cb_0\eta_{j,n+1}(\lambda), n\geq c,
\end{cases}
\end{equation}}
and for $j\geq c$,{\small
\begin{equation}\label{22MtLaplace forward equation
for Mt}
\begin{cases}
(\lambda+h)\eta_{j0}(\lambda)=b_0\eta_{j1}(\lambda),\\
[\lambda-b_1+(n-1)b_0+\beta]\eta_{jn}(\lambda)=h_n\eta_{j0}(\lambda)+\sum\limits_{k=1}^{n-1}b_{n-k+1}\eta_{jk}
(\lambda)\\
\hspace{5.5cm}+(n+1)b_0\eta_{j,n+1}(\lambda),\ \ 1\leq n\leq c-1,\\
[\lambda-b_1+(c-1)b_0+\beta]\eta_{jn}(\lambda)-\delta_{jn}=h_n\eta_{j0}(\lambda)
+\sum\limits_{k=1}^{n-1}b_{n-k+1}\eta_{jk}
(\lambda)\\
\hspace{6.5cm}+cb_0\eta_{j,n+1}(\lambda),\ \ \ n\geq c.
\end{cases}
\end{equation}}
Let
\begin{equation}\label{etajnetajnlambda}
\eta_{jn}(\lambda)=D_j(\lambda)r_{jn}(\lambda)+F_j(\lambda)r_{0n}(\lambda),\
\ j\geq1,\ \ n\geq 0.
\end{equation}
Recalling the Laplace transform of \eqref{the forward equations of
Nt}, and substituting \eqref{etajnetajnlambda} into
\eqref{11MtLaplace forward equation for Mt}, \eqref{22MtLaplace
forward equation for Mt} respectively, one has
\begin{equation}\label{BlambdaClambda}
D_j(\lambda)=1,\ \ F_j(\lambda)=\frac{\beta[\lambda
r_{j0}(\lambda)-1]}{\lambda+\beta-\lambda\beta r_{00}(\lambda)},\ \
j\geq1.
\end{equation}
Hence, making use of \eqref{link between pi and hatpi} in
\eqref{etajnetajnlambda}, with $D_j(\lambda)$ and $F_j(\lambda)$
given in \eqref{BlambdaClambda}, some straightforward calculations
lead us to \eqref{jnetajn} for $j\geq1$. Furthermore, taking the
Laplace transform of the first equation in \eqref{Mtforward equation
for Mt}, we obtain
$\eta_{j,-1}(\lambda)=\frac{\beta}{\lambda+\beta}\Big[\frac{1}{\lambda}-\eta_{j0}(\lambda)\Big]$.
Hence, making use of \eqref{jnetajn} for $n=0$, \eqref{etaj-1}
finally follows. The proof is complete.
\end{proof}
\par
Let us now denote by $\Delta_{j0}(\lambda)$ the Laplace transform of
$d_{j0}(t)\ (j\in {\bf{E}})$. Since \eqref{link between cj0 and Mt}
implies $d_{j0}(t)=h'_{j,-1}(t)\ (j\in {\bf{E}})$, i.e.,
$\Delta_{j0}(\lambda)=\lambda\eta_{j,-1}(\lambda)$, recalling
\eqref{etaj-1} for $j\in {\bf{E}}$ and $\lambda>0$, we immediately
have the following corollary.
\begin{corollary}\label{deltajolamdadeltajolamda}
For all $j\in {\bf{E}}$, there holds
\begin{equation}\label{deltajolamda}
\Delta_{j0}(\lambda)=\frac{\beta}{\lambda+\beta}-\frac{\lambda}{\lambda+\beta}\cdot
\frac{\beta\tilde{r}_{j0}(\lambda+\beta)}{1-\beta\tilde{r}_{00}(\lambda+\beta)},
\end{equation}
where $\tilde{r}_{j0}(\lambda+\beta)\ (j\in {\bf{E}})$ are given by
Remark~\emph{\ref{remark 3.1}}.
\end{corollary}
\par
 Next we shall study the expectation and variance of the first catastrophe time
 $C_{j0}$. Since
$E(C_{j0})=-\lim_{\lambda\rightarrow0^+}\Delta_{j0}^\prime(\lambda)$
and
$E(C_{j0}^2)=\lim_{\lambda\rightarrow0^+}\Delta_{j0}^{\prime\prime}(\lambda)$,
by \eqref{deltajolamda} and after some easy algebraic computation,
we can immediately obtain the following theorem.
\begin{theorem}
For all $j\in {\bf{E}}$, there holds {\small
\begin{align}
\label{ECj0}
E(C_{j0})&=\frac{1}{\beta}+\frac{\tilde{r}_{j0}(\beta)}{1-\beta\tilde{r}_{00}(\beta)},\\
\label{VarCj0}
Var(C_{j0})&=\frac{1}{\beta^2}\bigg\{1-\frac{\beta^2\tilde{r}_{j0}^2(\beta)}{[1-\beta\tilde{r}_{00}(\beta)]^2}
-\frac{2\beta^2}{1-\beta\tilde{r}_{00}(\beta)}\frac{\emph{d}}{\emph{d}\beta}\tilde{r}_{j0}(\beta)
-\frac{2\beta^3\tilde{r}_{j0}(\beta)}{[1-\beta\tilde{r}_{00}(\beta)]^2}\frac{\emph{d}}{\emph{d}\beta}\tilde{r}_{00}
(\beta)\bigg\},
\end{align}}
where $\tilde{r}_{j0}(\beta)\ (j\in {\bf{E}})$ are given by
Remark~\emph{\ref{remark 3.1}}.
\end{theorem}
\par
The following theorems will show the asymptotic behavior of the mean
first catastrophe time for $\beta\downarrow0$ and for
$\beta\rightarrow+\infty$.
\begin{theorem}
\emph{(i)} If $B_c^\prime(1)<0$ and $\mu_1<\infty$, then
\begin{equation}\label{xi0ECj0}
\lim_{\beta\downarrow0}\beta
E(C_{j0})=\frac{-B_c^\prime(1)+\mu_1+\sum_{k=1}^{c-1}r_k(c-k)b_0}
{\mu_1+\sum_{k=1}^{c-1}r_k(c-k)b_0},\ \ \ j\in {\bf{E}};
\end{equation}
\par
 \emph{(ii)} if $B_c^\prime(1)>0$, then
\begin{align}
\label{Ec00-1beta}
\lim_{\beta\downarrow0}\Big\{E(C_{00})-\frac{1}{\beta}\Big\}&=\frac{1}{h-H(u)+\sum_{k=1}^{c-1}u^k(c-k)b_0(1-u)r_k},\\
\label{Ecj0-1beta}
\lim_{\beta\downarrow0}\Big\{E(C_{j0})-\frac{1}{\beta}\Big\}&=\frac{b_0m_1^*(j)}{h-H(u)+\sum_{k=1}^{c-1}u^k(c-k)
b_0(1-u)r_k},\ \ \ j\geq1,
\end{align}
where $(m_1^*(j);\ j\geq1)$ and $(r_k;\ 1\leq k\leq c-1)$ are given
by \eqref{eklinearequations} and \eqref{rrklinearequations},
respectively.
\end{theorem}
\begin{proof} (i) If $B_c^\prime(1)<0$ and $\mu_1<\infty$,
then from Theorem~\ref{abcthth 3.2} we know that $\{\tilde{X}_t;\
t\geq0\}$ is positive recurrent. Hence, using \eqref{ECj0} and
Tauberian theorem, we have
\begin{equation*}
\lim_{\beta\downarrow0}\beta
E(C_{j0})=1+\frac{\lim_{t\rightarrow+\infty}\tilde{p}_{j0}(t)}{1-\lim_{t\rightarrow+\infty}\tilde{p}_{00}(t)}
=\frac{1}{1-\tilde{\pi}_0}.
\end{equation*}
Then, \eqref{xi0ECj0} immediately follows from \eqref{bbb 3.11}.
\par
 (ii) If $B_c^\prime(1)>0$, then $\{\tilde{X}_t;\
t\geq0\}$ is transient. Hence, by \eqref{bbb 2.5},
\eqref{tildetilder00lambdar00lambda},
\eqref{tildetilderi0lambdari0lambda} and \eqref{ECj0}, we can obtain
\eqref{Ec00-1beta} and \eqref{Ecj0-1beta}. The proof is complete.
 \end{proof}
\begin{theorem}
For $E(C_{j0})$ given by \eqref{ECj0}, there holds
\begin{align*}
\lim_{\beta\rightarrow+\infty}E(C_{00})&=\frac{1}{h},\\
\lim_{\beta\rightarrow+\infty}\beta E(C_{j0})&=
\begin{cases}
1+\frac{b_0}{h},&\ \ \ {j=1,}\\
1,&\ \ \ {j\geq2}.
\end{cases}
\end{align*}
\end{theorem}
\begin{proof} Since
\begin{equation*}
(\beta+h)\tilde{r}_{j0}(\beta)=\delta_{j0}+b_0\tilde{r}_{j1}(\beta),\
\ j\in {\bf{E}},
\end{equation*}
from \eqref{ECj0} we have
\begin{align*}
E(C_{00})&=\frac{1}{\beta}+\frac{\tilde{r}_{00}(\beta)}{h\tilde{r}_{00}(\beta)-b_0\tilde{r}_{01}(\beta)},\\
E(C_{j0})&=\frac{1}{\beta}\bigg[1-\frac{h\tilde{r}_{j0}(\beta)-b_0\tilde{r}_{j1}(\beta)}{h\tilde{r}_{00}
(\beta)-b_0\tilde{r}_{01} (\beta)}\bigg],\ \ \ (j=1,2,\ldots).
\end{align*}
By Tauberian theorem, we obtain
\begin{equation*}
\lim_{\beta\rightarrow+\infty}E(C_{00})=\frac{\lim_{t\rightarrow0}\tilde{p}_{00}(t)}
{h\lim_{t\rightarrow0}\tilde{p}_{00}(t)-b_0\lim_{t\rightarrow0}\tilde{p}_{01}(t)}=\frac{1}{h}
\end{equation*}
and
\begin{align*}
\lim_{\beta\rightarrow+\infty}\beta
E(C_{j0})&=1-\frac{h\lim_{t\rightarrow0}\tilde{p}_{j0}(t)-b_0\lim_{t\rightarrow0}\tilde{p}_{j1}(t)}
{h\lim_{t\rightarrow0}\tilde{p}_{00}(t)-b_0\lim_{t\rightarrow0}\tilde{p}_{01}(t)}\\&=
\begin{cases}
1+\frac{b_0}{h},&\ \ \ {j=1,}\\
1,&\ \ \ {j\geq2}.
\end{cases}
\end{align*}
The proof is complete.
 \end{proof}
\section*{Acknowledgements}
This work was substantially supported by the National Natural
Sciences Foundations of China (No. 11371374, No. 11571372), Research
Fund for the Doctoral Program of Higher Education of China (No.
20110162110060).


\end{document}